\documentclass[12pt, reqno]{amsart}
\usepackage[utf8]{inputenc}
\usepackage{a4wide,amsfonts,amsfonts, amsmath,amssymb,amsthm,epsfig,cite,graphicx,hyperref,
color, esint,fancyhdr, enumerate, latexsym,amsrefs, makecell, float, blindtext, multicol}
\usepackage{calligra}
\usepackage{enumitem}
\usepackage{tcolorbox}

\newtheorem{thm}{Theorem}[section]
\newtheorem{cor}[thm]{Corollary}

\numberwithin{equation}{section}

\newcommand{\mbb}{\mathbb}
\newcommand{\ra}{\rightarrow}

\newcommand{\sm}{\setminus}
\newcommand{\ep}{\epsilon}
\newcommand{\no}{\noindent}

\newcommand{\Om}{\Omega}
\newcommand{\ti}{\tilde}

\newcommand{\De}{\Delta}

\title{A quasiconformal variant of the union problem}
\author{Diganta Borah, Prachi Mahajan, Kaushal Verma}

\address{DB: Indian Institute of Science Education and Research Pune, Pune  411008, India}
\email{dborah@iiserpune.ac.in}

\address{PM: Department of Mathematics, Indian Institute of Technology Bombay, Powai, Mumbai 400076, India}
\email{prachi.mjn@iitb.ac.in}

\address{KV: Department of Mathematics, Indian Institute of Science, Bengaluru-560012, India}
\email{kverma@iisc.ac.in}

\begin{document}

\begin{abstract}
   The Union Problem, which has its genesis in the classical Levi problem, asks for a classification of complex manifolds $M$ that can be exhausted by an increasing union of submanifolds $M_j \subset M$ which are all biholomorphic to a fixed domain in $\mbb C^n$. We explore a quasiconformal variant of this question and seek to classify $n$-Riemannian manifolds $M$ such that each $M_j$ is quasiconformally equivalent to a bounded domain in $\mbb R^n$. It turns out that this is possible when these quasiconformal equivalences have uniformly bounded dilatations. Using Kiernan's quasiconformal Schwarz lemma when $n=2$ and Ferrand's conformal capacity when $n \ge 3$, we classify a class of $n$-Riemannian manifolds $M$ such that each $M_j$ is $K_j$-quasiconformally equivalent to $\Om \sm A$, where $\sup K_j < \infty$ and $\Om \subset \mbb R^n$ is a $C^2$-smoothly bounded domain and $A \subset \Om$ is at most finite. As a consequence, we obtain that Gehring's example of a bounded domain in $\mbb R^n$ which has $C^1$-smooth boundary everywhere except at a point and is known to be quasiconformally inequivalent to the unit ball in $\mbb R^n$, possesses the additional property that it cannot even be exhausted by quasiconformal images of the unit ball with uniformly bounded dilatations.
\end{abstract}


\maketitle

\section{Introduction}

\no Let $M$ be a complex manifold exhausted by an increasing sequence of open Stein submanifolds $M_j$, i.e., $M_1 \Subset M_2 \ldots \Subset M = \cup M_j$. In this setting, Behnke and Thullen \cite{BT} asked the following question in $1933$: Is $M$ Stein? Behnke--Stein \cite{BS39} provided a positive answer to this question when $M$ is an open subset of $\mbb C^n$, $n \ge 1$ by showing that the union of an increasing sequence of domains of holomorphy is a domain of holomorphy. While being an important ingredient in the solution to the classical Levi problem, its relevance also stems from the general fact that the union of a pair of domains of holomorphy need not always be a domain of holomorphy. On the other hand, Forn\ae ss provided explicit counterexamples (\cite{For77}, \cite{For76}) which established that
the answer to the question of Behnke--Thullen is generally negative. In dimension one, the situation is markedly different; the above question reduces to a purely topological problem due to the following result of Behnke--Stein \cite{BS49} which states that a Riemann surface is Stein if and only if it is non-compact. This provides an affirmative answer to the above question in dimension one. The counterexamples due to Forn\ae ss were the reason behind the study of the following variant of this problem: Let $M$ be a complex manifold which is exhausted by an increasing sequence of open subsets $M_j$, $j \ge 1$, each of which is biholomorphic to a given domain $\Om \subset \mbb C^n$. Determine $M$ in terms of $\Om$. This problem is sometimes called the Union Problem in the literature. 

\medskip

We are interested in exploring a quasiconformal variant of this question. More precisely, if $M$ is an $n$-dimensional Riemannian manifold that is exhausted by an increasing sequence of submanifolds $M_j$, $j \ge 1$, and each $M_j$ is quasiconformally equivalent to a fixed $n$-dimensional Riemannian manifold $\Om$, can $M$ be described in terms of $\Om$? This is indeed the case if the dilatations of the quasiconformal equivalences between $\Om$ and $M_j$ are uniformly bounded. A stronger conclusion is obtained if these dilatations converge to $1$. The main results for cases $n = 2$ and $n \ge 3$ are contained in Section $2$ and Section $3$, respectively. Finally, it must be noticed that the techniques used in Section $3$ are also applicable in Section $2$ and the different choice that has been made is entirely to highlight the diversity of techniques that are available to address such questions.

\medskip

Part of the reason for considering this quasiconformal variant of the Union Problem stems from recalling a well known phenomenon, namely, that the union of an increasing sequence of quasidisks is not always a quasidisk. Indeed, if the quasidisks have uniformly bounded dilatation, the limit will be a quasidisk, but when the dilatations grow unboundedly, the limit domain could have a non-Jordan boundary or develop inward cusps, in which case, the Ahlfors three-point condition would be violated, and this would prevent the limit from being a quasidisk.

\section{Case $n=2$}

In this case $M$ is a Riemann surface that is exhausted by quasiconformal images of a fixed domain. Two ingredients are needed to analyse the possibilities for $M$. 

\medskip

First, let $\mathcal{F}$ be a family of 
$K$-quasiconformal maps of $M$ into $\mbb P^1$ equipped with the spherical metric. If there exists a $d>0$ such that each $f \in \mathcal{F}$ omits two points in $\mbb P^1$ whose spherical distance is at least $d$, then the family $\mathcal{F}$ is normal. Indeed, note that Lemma~5.1 and consequently Lemma~5.2 of \cite{LV} hold when the domain $G$ is replaced by a Riemann surface $M$. Thus, we only need to establish that the family $\mathcal{F}$ is equicontinuous on $M$. This can be shown by rewriting the proof of Theorem~4.1 of \cite{LV} on a coordinate disc centered at $z_0 \in M$. Thus, an analogue of Theorem~5.1-(1) in \cite{LV} holds for families of quasiconformal maps with uniformly bounded dilatations from a Riemann surface to $\mbb P^1$ equipped with the spherical metric.

\medskip

Second, let $d_M$ denote the Kobayashi pseudo-distance on $M$. If $M$ is hyperbolic, then $d_M$ is a distance, but if $M$ is non-hyperbolic, then $d_M \equiv 0$. It is well known that holomorphic maps are distance decreasing with respect to the Kobayashi distance. We recall an analogue of this for quasiconformal maps from Kiernan \cite{K}. For each $K \geq 1$ and $0<\delta \leq 32^{-K}$, a pseudodistance $h_{M,K}=h_{M,K,\delta}$ was introduced in \cite{K} as follows: first note that by Theorem~2 in \cite{K}, there exists a constant $C_{K,\delta}>0$ such that for each holomorphic self-map $f: \Delta \to \Delta$ of the unit disc $\Delta \subset \mbb C$,
\[
d_{\Delta}\big(f(x),f(y)\big) \leq
\begin{cases}
C_{K,\delta}\big(d_{\Delta}(x,y)\big)^{1/K},  &  \text{if } d_{\Delta}(x,y)\leq \delta\\
C_{K,\delta}d_{\Delta}(x,y), & \text{if } d_{\Delta}(x,y) > \delta,
\end{cases}
\]
for all $x,y \in \Delta$. Set
\[
h_{M,K,\delta}(z_1,z_2)=
\begin{cases}
C_{K,\delta}d_M(z_1,z_2), & \text{if $d_M(x,y)\geq 1$}\\
C_{K,\delta}\big(d_M(z_1,z_2)\big)^{1/K}, & \text{if $d_M(x,y)< 1$}.
\end{cases}
\]
Then $h_{M,K,\delta}$ is a pseudo-distance on $M$ that depends on $K,\delta$ and $C_{K,\delta}$. Observe that $h_{M,K,\delta}\equiv 0$ if $d_M\equiv 0$ and $h_{M,K,\delta}$ is a distance otherwise. When $\delta$ and $C_{K,\delta}$ are fixed, $h_{M,K,\delta}$ will be denoted by $h_{M,K}$. In \cite{K} it was shown that if 
$f: M \to M'$ is a $K$-quasiconformal map between Riemann surfaces $M, M'$, Then $f$ is 
distance decreasing with respect to the pseudo-distances $h_{M,K}$ and $d_{M'}$, i.e.,
\[
d_{M'}\big(f(x),f(y)\big) \leq h_{M,K}(x,y)
\]
for all $x,y \in M$.

\medskip

\noindent All theorems that follow require only a uniform bound on the dilatations $K_j$, but the statements of all theorems below indicate the evidently clear consequence that $K_j \ra K_0$ after possibly passing to a subsequence.

\begin{thm} \label{T1}
Let $ M $ be a hyperbolic Riemann surface exhausted by an increasing sequence of open subsets $M_j \Subset M_{j+1}$, $j \ge 1$, and let $\Omega \subset \mbb C$ be a bounded domain with $C^2$-smooth boundary. Assume that there exists a $K_j$-quasiconformal homeomorphism $\phi_j: \Omega \rightarrow M_j$ for each $j \ge 1$. If $K_j \to K_0 < \infty$, then $M$ is $K_0$-quasiconformally equivalent to either $\Omega$ or the unit disc. In particular, if $ K_j \rightarrow 1 $, then  $ M $ is biholomorphic either to $ \Omega $ or to the unit disc $ \Delta $.
\end{thm}

\begin{proof}
The dichotomy that $M$ is quasiconformally equivalent to either $\Omega$ or $\Delta$ arises from the following. Set
\[
\psi_j := \phi_j^{-1}: M_j \rightarrow \Omega.
\]
Then $\psi_j$ is a $K_j$-quasiconformal homeomorphism of $M_j$ onto $\Omega$. Fix $ z_0 \in M $. Since $M_j \Subset M_{j+1}$ is an exhaustion of $M$, we may assume that $ z_0 \in M_j $ for all $ j$. There are two cases to be considered:
\begin{enumerate}
 \item[(a)] $\{\psi_j(z_0)\} $ is a relatively compact subset of $\Omega$.
 \item[(b)] $\{\psi_j(z_0)\} $ has at least one limit point in $\partial \Omega $.
\end{enumerate}
In Case (a), $M$ will turn out to be $K_0$-quasiconformally equivalent to $\Omega$, while in case (b), it will be $K_0$-quasiconformally equivalent to 
$\Delta$. Before proceeding further, note that $\{K_j\}$ being a convergent sequence is bounded, say by $K$ and it is immediate that $\psi_j$ is $K$-quasiconformal for all $j \ge 1$.

\medskip

\noindent {\sf Case (a)}: There are several steps in the proof.

\medskip

\noindent {\it Step I}. A subsequence of $\{\psi_j: j \ge 1\}$ converges uniformly on compact subsets of $M$ to a $K_0$-quasiconformal map $\psi: M \to \Omega$.

\medskip

Let $U_n \Subset U_{n+1}$ be an exhaustion of $ M $. For each $ n $, there exists an $ N \ge 1$ such that $U_n \subset M_j$ for every $ j \geq N $. Consequently, for every $j \geq N$, $\psi_j$ is a $K$-quasiconformal map of $U_n$ into $\Omega$. Since $\Omega$ is bounded, 
$ \{\psi_j\vert_{U_n}: j \geq N\} $ forms a normal family. Hence, some subsequence of $ \{\psi_j: j \ge N\}$ converges uniformly on compact subsets of $ U_n $ to $ \psi^{(n)} : U_n \rightarrow \overline{\Omega}$. By passing to the diagonal subsequence, there exists a subsequence (which we continue to denote by the same symbols) of the sequence $\psi_j$ that converges uniformly on compact subsets of $M$ to $ \psi: M \rightarrow \overline{\Omega} $. Moreover, 
\begin{equation*}
 {\psi|}_{U_n}= \psi^{(n)}   
\end{equation*}
for each $n$. 

\medskip

The claim now is that $\psi$ is constant or $K_0$-quasiconformal. To see this, let $\epsilon>0$ be arbitrary. Since $K_j \to K_0$, $\psi_j$ is $(K_0+\epsilon)$-quasiconformal for all large $j$, and hence it follows from Theorem 7.3 of \cite{F1} that either $ \psi^{(n)} $ is constant on $ U_n $ or 
$ \psi^{(n)} $ is a $(K_0 + \epsilon)$ quasiconformal mapping of $ U_n $.

\medskip

If there exists $ n \ge 1$ such that $ \psi^{(n)} $ is constant, then for any $ m \geq n $, the restriction $ \psi^{(m)}\vert_{U_n} $, as it agrees with $ \psi^{(n)} $, must also be constant. This implies that $ \psi $ is constant on $ M $. Otherwise, $ {\psi|}_{U_n} =\psi^{(n)} $ is $ (K_0 + \epsilon) $-quasiconformal for every $ n $, and hence $ \psi $ is $ (K_0 + \epsilon) $-quasiconformal on $M $, since quasiconformality is a local property. Since $\epsilon>0$ is arbitrary, it follows that $\psi$ is $K_0$-quasiconformal. This proves the claim.

\medskip

Now we note that in either of the above possibilities, $\psi$ maps $M$ into $\Omega$. Indeed, if $\psi$ is constant, then $\psi\equiv\psi(z_0)$ and so $\psi: M \to \Omega$. On the other hand, if $\psi$ is quasiconformal, then $\psi:M \to \mbb P^1$ is a homeomorphism onto its image, and hence, by the invariance of domains, it is an open map. Therefore, $\psi(M)$ is an open subset $\mbb P^1$. Since  $\psi(M) \subset \overline\Omega $, we must have $\psi(M) \subset \Omega$ and, therefore, $\psi: M \rightarrow \Omega$.

\medskip

Now we rule out the possibility that $\psi$ is constant by showing that it is injective. Let $0<\delta<32^{-K}$ and $h_{\Omega, K}=h_{\Omega, K, \delta}$ be the pseudo-distance on $\Omega$ defined above. Let $ z_1, z_2 $ be any two points in $ M $. Then for each $j$,
\begin{equation}\label{Kob-iso}
d_M(z_1,z_2) = d_M \big( \phi_j \circ \psi_j (z_1), \phi_j \circ \psi_j (z_2) \big)
\end{equation}
and since $ M_j \subset M $, it follows that
\begin{equation}\label{Kob-dec}
d_M \big( \phi_j \circ \psi_j (z_1), \phi_j \circ \psi_j (z_2) \big) \leq d_{M_j} \big( \phi_j \circ \psi_j (z_1), \phi_j \circ \psi_j (z_2) \big).  
\end{equation}
The $K$-quasiconformal map $ \phi_j : \Omega \rightarrow M_j $ is distance decreasing with respect to the metrics $ h_{\Omega, K} $ and $ d_{M_j} $, and hence
\begin{equation}\label{quasi-dec}
d_{M_j} \big( \phi_j \circ \psi_j (z_1), \phi_j \circ \psi_j (z_2) \big) 
\leq h_{\Omega, K}  \big( \psi_j(z_1), \psi_j(z_2) \big)
\end{equation}
and by using the triangle inequality, we get
\begin{alignat}{3} \label{0.3} 
h_{\Omega,K}  \big( \psi_j(z_1), \psi_j(z_2) \big) 
\leq 
h_{\Omega,K}  \big( \psi_j(z_1), \psi(z_1) \big) + h_{\Omega, K}  \big( \psi(z_1), \psi(z_2) \big)  \nonumber \\
+ \; h_{\Omega, K}  \big( \psi_j(z_2), \psi(z_2) \big). 
\end{alignat}
Combining (\ref{Kob-iso})--(\ref{0.3}) and letting $ j \rightarrow \infty$ gives
\[
d_M(z_1,z_2) \leq h_{\Omega, K}  \big( \psi(z_1), \psi(z_2) \big).
\]
Since $ M $ is hyperbolic, we must have $ z_1 = z_2 $ whenever $ \psi(z_1) = \psi(z_2) $. It follows that $ \psi $ is injective on $ M $. In particular, $ \psi $ is nonconstant, and hence, by the preceding dichotomy, $ \psi $ is a $K_0$-quasiconformal on $ M $.

\medskip

\noindent {\it Step II}. A subsequence of $\{\phi_j: j \ge 1\}$ converges uniformly on compact subsets of $\Omega$ to a $K_0$-quasiconformal map $\phi: \Omega \to M$.

\medskip

Unlike in the previous step, it is not evident that the family $\{\phi_j\}$ is normal since the maps $\phi_j$ take values in $M$. However, the previous step shows that $M$ is quasiconformally equivalent to the bounded domain $M'=\psi(M) \subset \Omega$, and we exploit this fact to establish the normality of $\{\phi_j\}$. Note that
\[
\Phi_j=\psi\circ \phi_j : \Omega \to M'.
\]
are $KK_j$-quasiconformal mappings of the fixed domain 
$ \Omega $ into a bounded domain and, therefore, the family $ \{\Phi_j\} $ is normal. Consequently, there exists a subsequence of $ \{\Phi_j\} $ (which we continue to denote by the same symbols) that converges uniformly on compact subsets of $ \Omega $ to $ \Phi: \Omega \rightarrow \overline{M'} $. By Theorem~7.3 of \cite{F1}, the limit map $ \Phi $ is either constant or $K_0^2$-quasiconformal on $ \Omega $. 

\medskip

The claim now is that $\Phi$ maps $\Omega$ into $M'$. First, observe that $\Phi$ fixes $\psi(z_0)$. Indeed,
\[
\Phi_j\big(\psi_j(z_0)\big)= \psi \circ \phi_j \big(\psi_j(z_0)\big)=\psi(z_0), \quad j \geq 1.
\]
Since $\psi_j(z_0) \to \psi(z_0)$ and $\Phi_j \to \Phi$ uniformly on compact subsets of $\Omega$, it follows that the left hand side converges to $\Phi(\psi(z_0))$ and hence this limit must be equal to $\psi(z_0)$. Thus, the map $\Phi$ is the constant map $\Phi\equiv \psi(z_0)$ or $K_0^2$-quasiconformal. In the first case, evidently $\Phi: \Omega \to M$. In the second case, by invariance of domains, $\Phi: \Omega \to M'$ and this proves the claim.

\medskip

It now follows that $\phi_j=\psi^{-1}\circ \Phi_j$ converges uniformly on compact subsets of $\Omega$ to the map $\phi=\psi^{-1}\circ \Phi: \Omega \to M$. Since $\phi_j$ is $K_j$-quasiconformal, and $K_j \to K_0$, as in Step I, for every $\ep>0$, the map $\phi$ is constant or $K_0+\epsilon$-quasiconformal. It follows that either $\phi$ is constant or $K_0$-quasiconformal.

\medskip

To show that $\phi$ is $K_0$-quasiconformal, observe that $\phi(\psi(z))=z$ for all $z \in M$. Indeed, given $z \in M$, we have $z \in M_j$ for all large $j$, and so
\[
\phi_j\big(\psi_j(z)\big)=z.
\]
As $\psi_j(z) \to \psi(z) \in \Omega$ and $\phi_j \to \phi$ uniformly on compact subsets of $\Omega$, it follows that the left-hand side converges to $\phi(\psi(z))$ and hence this limit must be equal to $z$. This shows that $\phi$ is non-constant and hence is $K_0$-quasiconformal.

\medskip

\noindent {\it Step III}. $\psi$ is a homeomorphism of $M$ onto $\Omega$. 

\medskip

As in Step II, $\psi(\phi(w))=w$ for all $w \in \Omega$. It follows that $\psi: M \to \Omega$ and $\phi: \Omega \to M$ are inverses of each other. This proves that $\psi: M \to \Omega$ is a $K_0$-quasiconformal homeomorphism and this completes Case (a).

\medskip 

\no {\sf Case (b)}: By passing to a subsequence, assume that $p_j=\psi_j(z_0) \to p_0 \in \partial \Omega $. In this case, the $\psi_j$'s converge locally uniformly on $M$ to the constant map $\psi_{\infty} \equiv p_0$. Therefore, every compact set in $M$ is eventually mapped near $p_0$ by $\psi_j$ for $j$ large. Let $r$ be a local $C^2$-smooth defining function for $\Om$ near $p_0$. Scaling $\Om$ near $p_0$ by the affine transformations
\[
A_j(z) = \frac{z - \psi_j(z_0)}{-r(\psi_j(z_0)}, \; z \in \mbb C 
\]
which satisfy $ A_j(\psi_j(z_0)) = 0 $ for all $j $, yields a family of scaled domains $\Om_j = A_j(\Om)$
that converge in the local Hausdorff sense to the half-plane
\[
\Omega_{\infty}  = \{ z \in \mathbb{C}: \Re \big(\partial r (p_0)z \big) < 1/2\}.
\]
To simplify notation, write
\[
\tilde{\psi}_j := A_j \circ \psi_j
\quad \text{and} \quad
\tilde{\phi}_j := \phi_j \circ A_j^{-1}.
\]
Then $ \tilde{\psi}_j : M_j \to \Omega_j $ and $\tilde{\phi}_j : \Omega_j \to M_j $ are $ K_j$-quasiconformal homeomorphisms which are inverses of each other.

\medskip

\noindent {\it Step I}. A subsequence of $ \{ \tilde{\psi}_j: j \ge 1 \} $ converges uniformly on compact subsets of $M$ to a $K_0$-quasiconformal map $\ti \psi: M \to \Omega_{\infty}$.

\medskip

First, observe that $ \{ \tilde{\psi}_j: j \ge 1 \} $ is a normal family. Indeed, let $U_n \Subset U_{n+1} $ be the exhaustion of $ M $ as before. Choosing distinct points $ a,b \in \mathbb{C} \setminus \overline{\Omega}_{\infty} $, we have $ a,b
\notin \Omega_j $ for all sufficiently large $ j $. Hence, for each $ n \ge 1$, there exists $ N(n) $ such that $ \tilde{\psi}_j $ is defined on $ U_n $ for all $ j \geq N(n) $, and
\[
\tilde{\psi}_j(U_n) \subset \Omega_j \subset \mathbb{C} \setminus \{a,b\}.
\]
The family $ \{ \tilde{\psi}_j: j \geq N(n)\}$, is therefore normal on $ U_n $, and admits a subsequence that converges uniformly on compact subsets of $ U_n $.  As established earlier in the proof, a diagonal argument yields a subsequence (which we continue to denote by the same symbols) of $ \{\tilde{\psi}_j \} $ uniformly converging on compact subsets of $ M $ to a limit map $\tilde \psi:M \to {\overline \Omega}_{\infty}$. Moreover, the limit map $ \tilde\psi $ is either constant or $K_0$-quasiconformal on $ M $. 

\medskip

Since $\tilde{\psi}_j(z_0)=0 $ for every $ j$, we have $ \tilde{\psi}(z_0)=0 $. Therefore, if $\tilde{\psi} $ is constant, then $\tilde{\psi}\equiv 0$. Otherwise, $ \tilde{\psi} $ is quasiconformal and hence by the invariance of domains, $ \tilde{\psi}(M)$ is an open subset of $\mbb P^1$. Since $ \tilde{\psi}(M)\subset {\overline{\Omega}}_{\infty}$, it follows that $ \tilde{\psi}(M)\subset \Omega_{\infty} $. Thus, in either case, $ \tilde{\psi}: M \to \Omega_{\infty} $. 

\medskip

We now show that $ \tilde{\psi} $ is injective. The argument is analogous to that used in Case (a), the only additional ingredient being the following stability property 
\begin{equation} \label{1.4}
\limsup_{j\to\infty}
h_{\Omega_j,K}\bigl(\tilde{\psi}(z_1),\tilde{\psi}(z_2)\bigr)
\leq
h_{\Omega_{\infty},K}\bigl(\tilde{\psi}(z_1),\tilde{\psi}(z_2)\bigr),
\end{equation}
for all $ z_1,z_2 \in M $. The above estimate follows directly from the definition of $ h_{\Omega_j, K} $ and the corresponding stability result for the Kobayashi distance on $ \tilde{\psi}(M) $, namely
\[
\limsup_{j\to\infty} d_{\Omega_j}\bigl(\tilde{\psi}(z_1),\tilde{\psi}(z_2)\bigr)
\leq
d_{\Omega_{\infty}}\bigl(\tilde{\psi}(z_1),\tilde{\psi}(z_2)\bigr).
\]
See, for instance, Proposition 3.1 in \cite{BBMV}. In view of the stability property (\ref{1.4}), the argument of Case (a) adapts readily to show that $ \tilde{\psi} $ is injective. Indeed, 
\begin{equation}
d_M(z_1,z_2) = d_M \big( \tilde\phi_j \circ \tilde\psi_j (z_1), \tilde\phi_j \circ \tilde\psi_j (z_2) \big). \label{1.1}
\end{equation}
for $ z_1, z_2 \in M $ and $ j \ge 1$ and it follows from the distance non-increasing property of the Kobayashi distance under the inclusion $ M_j \hookrightarrow M $ that
\begin{equation*}
d_M \big( \tilde\phi_j \circ \tilde\psi_j (z_1), \tilde\phi_j \tilde\circ \psi_j (z_2) \big) \leq d_{M_j} \big( \tilde\phi_j \circ \tilde\psi_j (z_1), \tilde\phi_j \circ \tilde\psi_j (z_2) \big).  \end{equation*}
Since the $K $-quasiconformal mappings $ \tilde\phi_j : \Omega_j \rightarrow M_j $ are distance non-increasing with respect to the metrics $ h_{\Omega_j, K} $ and $ d_{M_j} $,
\begin{equation}
d_{M_j} \big( \tilde\phi_j \circ \tilde\psi_j (z_1), \tilde\phi_j \circ \tilde\psi_j (z_2) \big) 
\leq h_{\Omega_j, K}  \big( \tilde\psi_j(z_1), \tilde\psi_j(z_2) \big).\label{1.2}
\end{equation}
Applying the triangle inequality yields the following inequality:
\begin{alignat}{3} \label{1.3} 
h_{\Omega_j, K}  \big( \tilde\psi_j(z_1), \tilde\psi_j(z_2) \big) 
\leq 
h_{\Omega_j, K}  \big( \tilde\psi_j(z_1), \tilde\psi(z_1) \big) + h_{\Omega_j, K}  \big( \tilde\psi(z_1), \tilde\psi(z_2) \big)  \nonumber \\
+ \; h_{\Omega_j, K}  \big( \tilde\psi_j(z_2), \tilde\psi(z_2) \big).
\end{alignat}
Observe that 
\[
d_{\Omega_j}\bigl(\tilde{\psi}_j(z_i),\tilde{\psi}(z_i)\bigr)\to 0,
\qquad i=1,2.
\]
as $ j \rightarrow \infty $ (see, for instance, Lemma 4.3 of \cite{MV}). Hence, by the definition of $ h_{\Omega_j,K} $, it follows that 
\[
h_{\Omega_j,K}
\bigl(\tilde{\psi}_j(z_i),\tilde{\psi}(z_i)\bigr)
\to 0,
\qquad i=1,2.
\]
Combining the above observation with the inequalities (\ref{1.1}), (\ref{1.2}) and (\ref{1.3}) and letting $ j \rightarrow \infty$ gives
\[
d_M(z_1,z_2) \leq h_{\Omega_{\infty}, K}  \big( \tilde\psi(z_1), \tilde\psi(z_2) \big).
\]
Since $ M $ is hyperbolic, $ z_1 = z_2 $ is immediate whenever $ \tilde\psi(z_1) = \tilde\psi(z_2) $. It follows that $ \tilde\psi $ is injective in $ M $, thereby ruling out the possibility that $\tilde{\psi} $ is constant. It follows that $\tilde{\psi}: M \rightarrow \Omega_{\infty} $ is a quasiconformal map. 

\medskip

\noindent {\it Step II}. A subsequence of $\{\tilde \phi_j: j \ge 1\}$ converges uniformly on compact subsets of $\Omega_{\infty}$ to a $K_0$-quasiconformal map $\tilde \phi: \Omega_{\infty} \to M$.

\medskip

To establish the normality of the family $\{\tilde \phi_j\}$, we use the fact that $M$ is quasiconformally equivalent to a domain $\tilde M=\tilde \psi(M) \subset \Omega_{\infty}$. Note that
\[
\tilde \Phi_j = \tilde \psi \circ \tilde \phi_j: \Omega_j \to \tilde M.
\]
is a $K_0K_j$-quasiconformal map of $\Omega$ into a domain $\tilde M$ whose complement in $\mbb P^1$ has more than two points.  Moreover, since $\Omega_j\to\Omega_{\infty} $ in the local Hausdorff sense, every compact subset of $ \Omega_{\infty} $ is contained in $ \Omega_j $ for all sufficiently large $ j $. Hence, $ \tilde{\phi}_j $ are eventually defined on every compact subset of $\Omega_{\infty} $. Now, considering an exhaustion of $\Omega_{\infty}$ by relatively compact open subsets, and proceeding as in Step I of Case (a), we obtain a subsequence of $\{\tilde \Phi_j\}$, which we denote by itself, that converges uniformly on compact subsets of $\Omega_{\infty}$ to a map $\tilde \Phi: \Omega_{\infty} \to \overline{\tilde M}$. which is either a constant map or a quasiconformal map. Also, note that
\[
\tilde \Phi_j(0)=\tilde \psi \circ \tilde \phi_j(0)= \tilde \psi(z_0)
\]
and hence $\tilde \Phi(0)=\tilde \psi(z_0) \in \ti M$. Thus, whether $\tilde \Phi$ is constant or quasiconformal, we always have $\tilde \Phi:\Omega_{\infty} \to \tilde M$. It now follows that the maps $\tilde \phi_j=\tilde \psi^{-1} \circ \tilde \Phi_j$ converge uniformly on compact subsets of $\Omega_{\infty}$ to the map $\tilde \phi=\tilde \psi^{-1}\circ \Phi: \Omega_{\infty} \to M$. Since $\tilde \phi_j$ is $K_j$-quasiconformal and $K_j \to K_0$, as in Step I of Case (a), for every $\ep>0$, the map $\tilde \phi$ is constant or $(K_0+\epsilon)$-quasiconformal. It follows that either $\tilde \phi$ is constant or $K_0$-quasiconformal.

\medskip

Next, we show that $ \tilde{\phi} $ is non-constant. Observe that $\tilde\phi(\tilde\psi(z))=z$ for all $z \in M$. Indeed, given $z \in M$, we have $z \in M_j$ for all large $j$, and so
\[
\tilde\phi_j\big(\tilde\psi_j(z)\big)=z.
\]
As $\tilde\psi_j(z) \to \tilde\psi(z) \in \Omega_{\infty}$ and $\tilde \phi_j \to \tilde\phi$ uniformly on compact subsets of $\Omega_{\infty}$, it follows that the left-hand side converges to $\tilde\phi(\tilde\psi(z))$ and hence this limit must be equal to $z$. This shows that $\tilde\phi$ is non-constant and hence is $K_0$-quasiconformal.

\medskip

\noindent {\it Step III}. $ \tilde \psi: M \ra \Om_{\infty} $ is a $K_0$-quasiconformal homeomorphism.

\medskip

As in Step II, $\tilde\psi(\tilde\phi(w))=w$ for all $w \in \Omega_{\infty}$. It follows that $\tilde\psi: M \to \Omega_{\infty}$ and $\tilde\phi: \Omega_{\infty} \to M$ are inverses of each other. This shows that $\tilde\psi: M \to \Omega_{\infty}$ is a $K_0$-quasiconformal homeomorphism. Since $\Omega_{\infty}$ is conformally equivalent to $\Delta$, this completes Case (b).

\medskip

Finally, the particular case follows immediately if we take $K_0=1$, as any $1$-quasi-conformal map is conformal.
\end{proof}

It remains to consider the case when $ M $ is non-hyperbolic. First, note that the union $ M $ is necessarily non-compact, as each $ M_j $ in the exhaustion is non-compact. In this case, the universal covering surface of $ M $ is $\mathbb{C} $. The only Riemann surfaces with universal covering surface $ \mathbb{C} $ are $\mathbb{C}$, $ \mathbb{C} \setminus \{0\} $, and complex tori. The latter are compact, whereas $ M $ is non-compact. Hence, if $ M $ is a non-hyperbolic Riemann surface, then $ M $ is biholomorphic to $ \mathbb{C} $ or  $ \mathbb{C} \setminus \{0\} $.

\medskip

We now show that the description of $M$ in the Union Problem remains unchanged if finitely many points are removed from $\Om$.

\begin{thm} \label{iso}
Assume that in the Union Problem, $ M $ is a hyperbolic Riemann surface and $ \Omega = D \setminus A $, where $ D $ is a bounded domain in $ \mathbb{C} $ with $C^2$-smooth boundary and $ A \subset D $ is finite. Then $ M $ is $K$-quasiconformally equivalent to $ \Omega $ or to the unit disc $ \Delta $.
\end{thm}

\begin{proof} Let $ \phi_j : \Omega \rightarrow M_j $ be $ K$-quasiconformal homeomorphisms, let $ \psi_j = \phi_j^{-1} $ denote their inverses. Fix $ z_0 \in M $. We first show that the sequence $ \{ \psi_j(z_0) \} $ cannot have an accumulation point in $ A $. The argument proceeds in several steps. 

\medskip

Assume, on the contrary, that a subsequence of $ \{ \psi_j(z_0) \}  $ (which we continue to denote by the same notation) converges to a point $ p_0 \in A $. Let $ \Delta( p_0, r) $ be a sufficiently small disc centered at the point $ p_0 $ in $ D $ such that $ \Delta( p_0, r) \cap A = \{p_0\} $. We will denote the punctured disc $\De(p_0,r)\setminus \{p_0\}$ by $\De'(p_0,r)$. Let $ \sigma_j$ denote the circle centered at $ p_0 $ with radius $ | \psi_j(z_0) - p_0| $. Since $
\psi_j(z_0) \rightarrow p_0 $, the circles $ \sigma_j $ are contained in $ \Delta'( p_0, r) $ for all sufficiently large $ j $. Our first step is to prove that
\begin{equation} \label{E1}
 \phi_j(\sigma_j) \to z_0   
\end{equation}
by adapting the ideas from the proof of Lemma 2 of \cite{JK}. To establish (\ref{E1}),  note that, by restricting $ \phi_j $ to $ \Delta'( p_0, r) $, we obtain a $K$-quasiconformal mapping $ \phi_j : \Delta'( p_0, r) \rightarrow M_j \subset M $. This restricted mapping is distance decreasing with respect to the metrics $ h_{\Delta'( p_0, r), K} $ and $ d_{M} $. Thus, for every $ w \in \sigma_j $, 
\begin{equation} \label{E2}
d_{M} \big( \phi_j (w), \phi_j \left( \psi_j (z_0) \right) \big) 
\leq h_{\Delta'( p_0, r) , K}  \big( w, \psi_j(z_0) \big).
\end{equation}
Before going further, observe that since $ w$ and $\psi_j(z_0) $ lie on the circle $\sigma_j $, 
\begin{equation*}
d_{\Delta'( p_0, r)}  \big( w, \psi_j(z_0) \big)  \leq \ell_{hyp} (\sigma_j),   
\end{equation*}
where $ \ell_{hyp} (\sigma_j) $ denotes the hyperbolic length of $ \sigma_j $ in $ \Delta'( p_0, r) $. The explicit expression for the hyperbolic metric on the punctured disc forces $ \ell_{hyp} (\sigma_j) \to 0 $, which, in turn, implies that $ d_{\Delta'( p_0, r) }  \big( w, \psi_j(z_0) \big) \to 0 $ and consequently that 
$$ 
h_{\Delta'( p_0, r) , K}  \big( w, \psi_j(z_0) \big) \to 0. 
$$
Combining the above observation with (\ref{E2}) and using the identity $ \phi_j(\psi_j(z_0) ) = z_0 $, it follows that
\[
d_{M} \big( \phi_j (w), z_0 \big) \rightarrow 0.
\]  
Moreover, since the metric $ d_M $ induces the topology of the hyperbolic Riemann surface $ M $, we conclude from above that  $ \phi_j (w) \rightarrow z_0 $, thus establishing (\ref{E1}).

\medskip

The goal now is to show that $ \phi_j $ extends across the isolated boundary point $ p_0 $. To this end,  let $ U $ be a relatively compact simply connected neighbourhood of $ z_0 $ in $ M $. By  (\ref{E1}), 
\[
\phi_j(\sigma_j) \subset U \subset M_j
\]
for all sufficiently large $ j $. Since $ U $ is simply connected, the closed curve $ \phi_j(\sigma_j) $ is 
null homotopic in $ M_j $.
Since $ \sigma_j $ generates $ \pi_1 \big( \Delta'( p_0, r) \big) $, it follows that the induced homomorphism
\[
(\phi_j)_* : \pi_1 \big( \Delta'( p_0, r) \big) \rightarrow \pi_1 (M_j)
\]
is trivial for all sufficiently large $ j $. Thus, after discarding finitely many terms, we may assume that $ (\phi_j)_* $ is trivial for every $ j $.

\medskip

Since each $ M_j $ is an open subset of the hyperbolic Riemann surface $ M $, it is itself hyperbolic. Hence, there exists a holomorphic covering projection $ \pi_j : \Delta \rightarrow M_j $. The triviality of $ (\phi_j)_* $ ensures that $ \phi_j $ lifts to a map 
\[ \tilde{\phi_j} : \Delta'( p_0, r)  \rightarrow \Delta, 
\]
satisfying $ \pi_j \circ \tilde{\phi_j} = \phi_j $. Moreover, the lift $ \tilde{\phi_j} $ is a $ K$-quasiconformal mapping of $ \Delta'( p_0, r)  $ into $ \Delta $. Indeed, first note that the injectivity of $ \phi_j $ implies that $ \tilde{\phi_j} $ is injective. Furthermore, the covering projection $ \pi_j $ is locally biholomorphic. Hence, for every point of $ M_j $, the corresponding local inverse branches of $ \pi_j $ are biholomorphisms and hence $ 1$-quasiconformal. Consequently, composing  with these local inverse branches does not change the maximal dilatation. Since $ \pi_j \circ \tilde{\phi_j} = \phi_j $, and $ \phi_j $ is $ K$-quasiconformal, the same holds for $ \tilde{\phi_j} $.

\medskip

In this setting, Corollary 6.4.27 of \cite{GMP} guarantees the existence of the limit 
\[
\lim_{w\to p_0} \tilde{\phi_j} (w) = q_j,
\]
where $ q_j $ is an isolated boundary point of $ \tilde{\phi_j} \big( \Delta'( p_0, r) \big)$. Moreover, the extension $ \tilde{\Phi}_j $
of $ \tilde{\phi_j} $ to $ \Delta( p_0, r) $ defined at $p_0$ by $ \tilde{\Phi}_j ( p_0) = q_j $ is a $ K $-quasiconformal homemorphism from  $ \Delta( p_0, r) $ onto the domain 
\[
X_j := \tilde{\phi_j} \big( \Delta'( p_0, r)  \big) \cup \{q_j\} \subset \overline{\Delta}.
\]
Since $ X_j $ is an open subset of $ \mathbb{C} $ contained in $ \overline{\Delta} $, every point of $ X_j $ lies in $ \Delta $. In particular, $ q_j \in \Delta $. As a consequence, the composition $ \Phi_j := \pi_j \circ \tilde{\Phi}_j $ is well-defined on $ \Delta( p_0, r) $. By construction,
\[
\Phi_j|_{ \Delta'( p_0, r) } = \pi_j \circ \tilde{\Phi}_j |_{ \Delta'( p_0, r)  } = \pi_j \circ \tilde{\phi}_j = \phi_j,
\]
and
\[
\Phi_j(p_0) = \big( \pi_j \circ \tilde{\Phi}_j \big) (p_0) = \pi_j(q_j) \in M_j.
\]  
The local biholomorphicity of $ \pi_j $ implies, by the same arguments as above, that $ \Phi_j $ is $ K$-quasiregular and locally injective on $ \Delta(p_0, r) $. We now use $ \Phi_j $ to define an extension of $ \phi_j $ across $ p_0 $ to $ \Omega \cup \{p_0\} $ as follows.
\[
\hat{\phi}_j(w)=
\begin{cases}
\phi_j(w), & w\in \Omega,\\
\Phi_j(p_0), & w =p_0.
\end{cases}
\]
It is evident that $ \hat{\phi}_j $ is $ K$-quasiregular and locally injective on $ \Omega \cup \{p_0\} $. Hence $ \hat{\phi}_j $ is open, and therefore a local homeomorphism. Moreover, $ \hat{\phi}_j $ assumes the value $ \pi_j(q_j) $ at both $ p_0 $ and $ \psi_j( \pi_j(q_j)) \neq p_0 $. We now use this observation to derive a contradiction to the injectivity of $ \phi_j $ on $ \Omega $.

\medskip

For brevity, set $ b_j := \pi_j(q_j) \in M_j $ and $ a_j := \psi_j(b_j) \in \Omega $. Since $ \hat{\phi}_j $ is a local homeomorphism, there exist disjoint neighbourhoods $ U_{0,j} $ of $ p_0 $ and $ U_{1,j} $ of $ a_j $ such that $ \hat{\phi}_j |_{U_{0,j}} $ and $ \hat{\phi}_j |_{U_{1,j}} $ are homeomorphisms onto open neighbourhoods $ V_{0,j} $ and $ V_{1,j} $, respectively, of $ b_j $ in $M_j$.
Choose 
\[
c_j \in V_{0,j} \cap V_{1,j}, \; c_j \neq b_j.
\]  
Then there are unique points $ x_{0,j} \in U_{0,j} $ and $ x_{1,j} \in U_{1,j} $ such that 
\[
\hat{\phi}_j (x_{0,j}) = c_j = \hat{\phi}_j (x_{1,j}). 
\]
Since $ \hat{\phi}_j (p_0) = b_j \neq c_j $, it follows that $ x_{0,j} \neq p_0 $. Hence $ x_{0,j}, x_{1,j} \in \Omega $, where $\hat{\phi}_j = \phi_j $. Consequently,
\[
\phi_j( x_{0,j}) = \phi_j( x_{1,j}), 
\] 
and the injectivity of $ \phi_j $ yields $ x_{0,j} = x_{1,j} $, contradicting the fact that
$ U_{0,j} $ and $ U_{1,j} $ are disjoint. This contradiction shows that the sequence $ \{ \psi_j(z_0) \} $ cannot have an accumulation point in $ A $.

\medskip

We now return to the proof of the theorem. By the preceding argument, the sequence $ \{ \psi_j(z_0) \} $ can accumulate only on $ \partial D $ or at points of $ \Omega $. The remainder of the proof is identical to that of Theorem~\ref{T1}, from which it follows that $ M $ is $K$-quasiconformally equivalent to the unit disc in the former case and to $ \Omega $ in the latter. 
\end{proof}

In the special case where $ \Omega = \Delta  \setminus \{0\} $, the following stronger result holds.

\begin{thm}
Assume that in the union problem, $ M $ is a hyperbolic Riemann surface. If $ \Omega = \Delta  \setminus \{0\} $, then  $ M $ is biholomorphic  either to $ \Delta  \setminus \{0\} $ or to the unit disc $ \Delta $.
\end{thm}

\begin{proof} Following the arguments in the proof of Theorem 3 of \cite{P}, we first show that each $ M_j $ is biholomorphic to the punctured unit disc $ \Delta \setminus \{0\} $. To this end, let $ \phi_j: \Delta  \setminus \{0\}  \rightarrow M_j $ be $K$-quasiconformal homeomorphisms and let $ \mu_{\phi_j} $ denote the Beltrami coefficient of $ \phi_j $. Let $ \mu_j : \mathbb{C} \to \mathbb{C} $ be the measurable extension of $ \mu_{\phi_j} $ obtained by setting it equal to zero outside
$ \Delta \setminus \{0\} $, that is
\[
\mu_j(z)=
\begin{cases}
\mu_{\phi_j}(z), & \mbox{ for almost every }z \in \Delta \setminus \{0\}, \\
0, & \text{ otherwise}.
\end{cases}
\]
Then 
\[
\| \mu_j \|_{\infty} \leq \frac{K-1}{K+1} < 1.
\]
Hence, by the mapping Theorem, there exists a quasiconformal mapping $ h_j : \mathbb{C} \to \mathbb{C} $, normalized by $ h_j(0) = 0 $, whose complex dilatation $ \mu_{h_j} $ coincides almost everywhere with $ \mu_j $. 

\medskip

Now, consider the quasiconformal homeomorphism $ \phi_j \circ h_j^{-1} : h_j( \Delta \setminus \{0\}) \to M_j $ and let $ \mu_{\phi_j \circ h_j^{-1}} $ denote its Beltrami coefficient.  Following the transformation rule for the Beltrami coefficient under composition, we have 
\[
\mu_{\phi_j \circ h_j^{-1}} (w) = \frac{\mu_{\phi_j}(z) - \mu_{h_j} (z) }{1 - \mu_{\phi_j}(z) \overline{\mu_{h_j}(z)}} \left( \frac{\partial h_j(z)} {|\partial h_j(z)|}\right)^2, \; w = h_j(z),
\]
for almost all $ z $ in $ \Delta \setminus \{0\} $, and hence for almost all $ w \in h_j( \Delta \setminus \{0\} )$. Since $ \mu_{h_j}=\mu_j=\mu_{\phi_j} $ almost everywhere on $ \Delta \setminus \{0\} $, it follows that $ \mu_{\phi_j \circ h_j^{-1}} = 0 $ almost everywhere on $ h_j(\Delta \setminus \{0\}) $. Consequently, $ \phi_j \circ h_j^{-1} $ is $1$-quasiconformal or, equivalently, a biholomorphism of $ h_j(\Delta \setminus \{0\}) $ onto $ M_j $. Since $ h_j $ is a homeomorphism satisfying $ h_j(0) = 0 $, it follows immediately that 
\[
h_j ( \Delta \setminus \{0\} ) = h_j( \Delta) \setminus \{0\}.
\] Moreover, $ h_j( \Delta) \subsetneq \mathbb{C} $ is a simply-connected planar domain containing the origin. By the Riemann mapping theorem, there exists a biholomorphism $ \eta_j: \Delta \to h_j( \Delta) $  such that $ \eta_j(0) = 0 $. Consequently, $ \eta_j: \Delta \setminus \{0\} \to h_j( \Delta) \setminus \{0\}$ is also a biholomorphism. It follows that
\[
 \phi_j \circ h_j^{-1} \circ \eta_j : \Delta \setminus \{0\} \to M_j,
\]
is a biholomorphism from $ \Delta \setminus \{0\} $ onto $ M_j $. In other words, each $ M_j $ is biholomorphic to $ \Delta \setminus \{0\} $. The desired conclusion now follows from Theorem 1.1(ii) of \cite{BMM}.
\end{proof}

\section{Case $ n \geq 3 $}

\noindent Let $ X $ be a non-compact Riemannian manifold of class $C^1$ and of dimension $ n \geq 2 $. Following \cite{F5}, we recall the definitions of the conformal capacity of a compact subset of $ X $ and the function $ \mu_X $, and summarize its basic properties that will be needed in the sequel.

\medskip

For a compact set $ C \subset X $, the conformal capacity of $ C $ in $ X $ is defined by  
\[
\mbox{Cap}_X(C) = \inf_u I(u, X), 
\]
where the infimum is taken over all continuous functions $ u $ with compact support in $ X $, satisfying $ u = 1 $ on $ C $, such that $ u $ admits a generalized differential $  \nabla u $ with
\[
I(u,X) := \int_X | \nabla u|^n \; dv < + \infty.
\]
Here, $ dv $ denotes the Riemannian volume element, and $ | \nabla u| $ is the norm of the generalized differential induced by the Riemannian metric. 

\medskip

The function $ \mu_X: X \times X \rightarrow [0, \infty) $ is then defined by
\[
\mu_X (x,y) = \inf_{C} \mbox{Cap}_X(C),
\]
where the infimum is taken over all compact connected subsets $ C \subset X $ that contain $ x $ and $ y $ and consist of more than one point. From \cite{F5}, it is known that $ \mu_{X} $ is continuous on $ X \times X $, and $ \mu_{X}(x,x) = 0 $ for all $ x \in X $, that $\mu_X$ satisfies the triangle inequality and finally, if $ X' $ is an open submanifold of $ X$, then $ \mu_X(x,y) \leq \mu_{X'}(x,y) $ for all $ x, y \in X' $.

\medskip
 
Based on $ \mu_X $, Ferrand introduced the following classification of non-compact Riemannian manifolds into two complementary classes. $ X $ is said to be of Class I if $ \mu_X $ is identically zero. On the other hand, $ X $ is said to be of Class II if $ \mu_X(x,y) =0 $ if and only if $ x=y $. The Euclidean space $ \mathbb{R}^n $ is an example of a Class I manifold, whereas every bounded subdomain of $ \mathbb{R}^n $ is of Class II. Moreover, every open submanifold of a Class II manifold is again in Class II. Among other things, it was also proved in \cite[(8.1)]{F5} that Classes I and II are invariant under quasiconformal homeomorphisms. In dimension $ n=2$, Ferrand's classes coincide with the classical dichotomy for non-compact Riemann surfaces: Class I consists of the parabolic Riemann surfaces, whereas Class II consists of hyperbolic Riemann surfaces.

\medskip

The following normal family criterion for quasiconformal mappings from a Riemannian manifold into $ \hat{\mathbb{R}}^n = \mathbb{R}^n \cup \{ \infty \}  $ will be useful for our purposes.

\begin{thm}\label{normal2}
Let $ X $ be a $ C^1$-smooth $ n$-dimensional Riemannian manifold and $\mathcal{F}$ be a family of $K$-quasiconformal mapppings of $X$ into $ \hat{\mathbb{R}}^n $. Suppose that there exists $d>0$ such that every $f \in \mathcal{F}$ omits two points of $ \hat{\mathbb{R}}^n $ whose spherical distance is at least $d$. Then $\mathcal{F}$ is a normal family.
\end{thm}

\begin{proof}
Let $ \epsilon > 0 $ be fixed and $ x \in X $. Since $ X $ is $ C^1$-smooth, by \cite[Lemma 3.1(a)]{F4}, there exist a neighbourhood $ U $ of $ x $ and a $ (1 + \epsilon)$-bilipschitzian homeomorphism $ \theta $ of $ U $ onto a domain $ U' $ in $ \hat{\mathbb{R}}^n $. As a consequence, both $ \theta, \theta^{-1} $ are $ (1 + \epsilon)^{2(n-1)} $-quasiconformal. Thus, the family
\[
 \{ f \circ \theta^{-1}: f \in \mathcal{F} \}
\]
consists of $ (1 + \epsilon)^{2(n-1)} K $-quasiconformal mappings of $ U'$ into $ \hat{\mathbb{R}}^n $. Moreover, each mapping $ f \circ \theta^{-1} $ omits two points of $ \hat{\mathbb{R}}^n $ whose spherical distance is at least $d$. In other words, 
\[
f \circ \theta^{-1}: U' \rightarrow \hat{\mathbb{R}}^n ,
\]
$ f \in \mathcal{F} $ satisfies the hypothesis of Theorems 19.2 and 20.5 of \cite{V} and, therefore, this family is normal on $ U' $. It follows that the family $ \{ f\vert_{U}: f \in \mathcal{F} \} $
is normal on $ U $. Since $ x \in X $ was arbitrary, $ \mathcal{F} $ is locally normal on $ X $. Finally, choosing a countable exhaustion of $ X $ by relatively compact open sets and applying a diagonal argument, we conclude that $ \mathcal{F} $ is normal on $ X $.
\end{proof}

\medskip

\no Note that since each $M_j$ in the exhaustion of $M$ is non-compact, so is $M$. Therefore, Ferrand's classification applies, and $M$ is either of Class I or of Class II. Here is the main theorem of this section.

\begin{thm} \label{T2}
Let $ M $ be a $ C^1$-smooth $ n$-dimensional Riemannian manifold exhausted by an increasing sequence of open subsets $M_j \Subset M_{j+1}$ and let $ \Omega $ be a bounded domain in $ \mathbb{R}^n $. Assume that, for each $ j $, there exists a $K_j$-quasiconformal homeomorphism $\phi_j: \Omega \rightarrow M_j$, where $ K_j \rightarrow K_0 $.
\begin{enumerate}

\item [(i)] If $ M $ is of Class II and $ \Omega $ has $C^2$-smooth boundary, then  $ M $ is 
$K_0$-quasiconformally equivalent to $ \Omega $ or the unit ball in $ \mathbb{R}^n $.

\item [(ii)] If $ M $ is of Class I, then $ M $ is 
$K_0$-quasiconformally equivalent to a class I domain in $ \mathbb{R}^n $.
\end{enumerate}

\end{thm}

\noindent\textit{Proof of Theorem \ref{T2}(i).} We follow the proof of Theorem \ref{T1} verbatim, using the same notation. By Theorem \ref{normal2}, for each $ n \ge 1$, the family $ \{ \psi_j := \phi_j^{-1}\vert_{U_n}: j \geq N\}$ is normal. Hence, by a diagonal argument there exists a subsequence, still denoted by $ \{ \psi_j \} $, which converges uniformly on compact subsets of $ M $ to $ \psi: M \rightarrow \overline{\Omega} $. The injectivity of the limit map $ \psi $ in Case (a), namely when $ \psi(z_0) \in \Omega $, is proved using the metric $ \mu_M $ as follows.
Firstly, note that, under the assumptions of (a), it follows exactly as in the proof of Theorem \ref{T1} that $\psi: M \to \Omega$. 

\medskip

Now, fix $ x_1 $ and $ x_2$ any two points in $ M $. Then for each $j$,
\begin{equation}
\mu_M(x_1,x_2) = \mu_M \big( \phi_j \circ \psi_j (x_1), \phi_j \circ \psi_j (x_2) \big). \label{0.21}
\end{equation}
Since $ M_j \subset M $, it follows that
\begin{equation*}
\mu_M \big( \phi_j \circ \psi_j (x_1), \phi_j \circ \psi_j (x_2) \big) \leq \mu_{M_j} \big( \phi_j \circ \psi_j (x_1), \phi_j \circ \psi_j (x_2) \big).    
\end{equation*}
Now, applying Theorem 2.4(c) of \cite{F1} to the $K_j$-quasiconformal homeomorphisms $\phi_j: \Omega \rightarrow M_j$ gives the following result:
\begin{equation}
\mu_{M_j} \big( \phi_j \circ \psi_j (x_1), \phi_j \circ \psi_j (x_2) \big) 
\leq K_j \; \mu_{\Omega}  \big( \psi_j(x_1), \psi_j(x_2) \big) \label{0.22}.
\end{equation}
Moreover,
\begin{alignat}{3} \label{0.23} 
\mu_{\Omega}  \big( \psi_j(x_1), \psi_j(x_2) \big) 
\leq 
\mu_{\Omega}  \big( \psi_j(x_1), \psi(x_1) \big) + \mu_{\Omega}  \big( \psi(x_1), \psi(x_2) \big)  \nonumber \\
+ \; \mu_{\Omega}  \big( \psi(x_2), \psi_j(x_2) \big). 
\end{alignat}
Furthermore, 
\[
\mu_{\Omega}
\bigl({\psi}_j(x_i),{\psi}(x_i)\bigr)
\to 0,
\qquad i=1,2.
\]
Combining the above observation with (\ref{0.21}), (\ref{0.22}) and (\ref{0.23}) and letting 
$ j \rightarrow \infty$ gives
\[
\mu_M(x_1,x_2) \leq \mu_{\Omega}  \big( \psi(x_1), \psi(x_2) \big).
\]
Since $ M $ is of Class II, we must have $ x_1 = x_2 $ whenever $ \psi(x_1) = \psi(x_2) $. It follows that $ \psi $ is injective on $ M $. Therefore, in Case (a), following the remaining steps of the proof of Theorem \ref{T1}, we conclude that the limiting map $ \psi $ is a $ K_0 $-quasiconformal homeomorphism from $ M $ onto $ \Omega $.

\medskip

We now consider Case (b), when $ \psi(z_0) \in \partial \Omega $. After passing to a subsequence, assume that $ \psi_j(z_0) \to p_0 \in \partial \Omega $. As before, the $\psi_j$'s converge locally uniformly on $M$ to the constant map $\psi_{\infty} \equiv p_0$. Therefore, every compact set in $M$ is eventually mapped near $p_0$ by $\psi_j$ for $j$ large. Applying the affine dilations 
\[
A_j(x) = \frac{ x - \psi_j(z_0)}{ - r( \psi_j(z_0))}, \; x \in \mathbb{R}^n,
\] 
where $ r $ is a local defining function for $ \Omega $ near $ p_0 $, we obtain scaled domains $\Omega_j := A_j(\Omega) $. Since $ \psi_j(z_0) \to p_0 \in \partial \Omega $, the domains $ \Omega_j $ converge in the local Hausdorff sense to the half-space
\[
\Omega_{\infty}  = \{ x \in \mathbb{R}^n: \langle \nabla r (p_0), x \rangle < 1\}.
\]
Define 
\[
\tilde{\psi}_j := A_j \circ \psi_j: M_j \to \Omega_j 
\quad \text{and} \quad
\tilde{\phi}_j := \phi_j \circ A_j^{-1} : \Omega_j \to M_j,
\]
and note that $ \tilde{\psi}_j(z_0) = 0 \in \Omega_{\infty} $ for each $j$. Moreover, since each $ A_j $ is an affine similarity, $ \tilde{\psi}_j $ and $\tilde{\phi}_j $ are $ K_j$-quasiconformal homeomorphisms that are inverses of each other. Arguing exactly as in the proof of Theorem \ref{T1} and using Theorem \ref{normal2}, we infer that $ \{\tilde{\psi}_j \}$ is normal and obtain a limit map $ \tilde{\psi} : M \rightarrow \Omega_{\infty} $ which is constant or quasiconformal. The remaining step is to prove that $ \tilde{\psi} $
is injective. Once this is established, the proof that $ \tilde{\psi} $ is a $ K_0 $-quasiconformal homeomorphism of $ M $ onto $ \Omega_{\infty} $
proceeds analogously to that of Theorem \ref{T1}. 

\medskip

In order to establish that $ \tilde{\psi} $ is injective, fix points $ x_1 $ and $ x_2$ in $ M $ and note that 
\begin{equation}
\mu_M(x_1,x_2) = \mu_M \big( \tilde{\phi}_j \circ \tilde{\psi}_j (x_1), \tilde{\phi_j} \circ \tilde{\psi}_j (x_2) \big) \label{0.31}
\end{equation}
for each $j$. As $ M_j \subset M $, we deduce that
\begin{equation*}
\mu_M \big( \tilde{\phi}_j \circ \tilde{\psi}_j(x_1), \tilde{\phi}_j \circ \tilde{\psi}_j (x_2) \big) \leq \mu_{M_j} \big( \tilde{\phi}_j \circ \tilde{\psi}_j (x_1), \tilde{\phi}_j \circ \tilde{\psi}_j (x_2) \big).    
\end{equation*}
Now, an application of Theorem 2.4(c) of \cite{F1} to the $K_j$-quasiconformal homeomorphisms $\tilde{\phi}_j: \Omega_j \rightarrow M_j$ implies that 
\begin{equation}
\mu_{M_j} \big( \tilde{\phi}_j \circ \tilde{\psi}_j (x_1), \phi_j \circ \psi_j (x_2) \big) 
\leq K_j \; \mu_{\Omega_j}  \big( \tilde{\psi}_j(x_1), \tilde{\psi}_j(x_2) \big) \label{0.32}.
\end{equation}
By the triangle inequality,
\begin{alignat}{3} \label{0.33} 
\mu_{\Omega_j}  \big( \tilde{\psi}_j(x_1), \tilde{\psi}_j(x_2) \big) 
\leq 
\mu_{\Omega_j}  \big( \tilde{\psi}_j(x_1), \tilde{\psi}(x_1) \big) + \mu_{\Omega_j}  \big( \tilde{\psi}(x_1), \tilde{\psi}(x_2) \big)  \nonumber \\
+ \; \mu_{\Omega_j}  \big( \tilde{\psi}(x_2), \tilde{\psi}_j(x_2) \big)
\end{alignat}
holds for all $ j $ sufficiently large. Since 
\[
\tilde{\psi}_j(x_1) \rightarrow \tilde{\psi}(x_1) \in \Omega_{\infty},
\]
we may choose a $ r_1 > 0 $ such that the Euclidean ball $ B \bigl( \tilde{\psi}(x_1), r_1 \bigr) $ is relatively compactly contained in $ \Omega_{\infty} $ and $\tilde{\psi}_j(x_1) $ are contained in $ B \big( \tilde{\psi}(x_1), r_1 \big) $ for all sufficiently large $ j $. Moreover, since $ \Omega_j \to \Omega_{\infty} $ in the local Hausdorff sense, it follows that $ B \bigl( \tilde{\psi}(x_1), r_1 \bigr) $ is relatively compactly contained in $ \Omega_j $ for all sufficiently large $ j $. In this setting, the monotonicity property of $\mu_X$ gives
\[
\mu_{\Omega_j}
\bigl(\tilde{\psi}_j(x_1), \tilde{\psi}(x_1)\bigr)
\leq \mu_{B \bigl( \tilde{\psi}(x_1), r_1 \bigr) } \bigl(\tilde{\psi}_j(x_1), \tilde{\psi}(x_1)\bigr),
\]
for all sufficiently large $ j $. Further, by the properties of $\mu_X$, the expression on the right side above tends to zero, and hence
\[
\mu_{\Omega_j}
\bigl(\tilde{\psi}_j(x_1),\tilde{\psi}(x_1)\bigr)
\to 0.
\]
Similarly, it can be proved that 
\[
\mu_{\Omega_j}
\bigl(\tilde{\psi}(x_2),\tilde{\psi}_j(x_2)\bigr)
\to 0.
\]
To control $ \mu_{\Omega_j}  \big( \tilde{\psi}(x_1), \tilde{\psi}(x_2) \big) $ as $ j \rightarrow \infty $, first choose an open set $ V $ such that 
\[
\{  \tilde{\psi}(x_1), \tilde{\psi}(x_2) \} \subset V \subset \overline{V} \subset \Omega_{\infty}.
\]
As before, the local Hausdorff convergence $ \Omega_j \rightarrow \Omega_{\infty} $ implies that 
$ V \subset \Omega_j $ for all sufficiently large $ j $. Invoking the monotonicity properties of $\mu_X$ again, we obtain
\[
\mu_{\Omega_j}  \big( \tilde{\psi}(x_1), \tilde{\psi}(x_2) \big) \leq \mu_{V}  \big( \tilde{\psi}(x_1), \tilde{\psi}(x_2) \big),
\]
for all sufficiently large $ j $. Finally, combining the above observation with (\ref{0.31}), (\ref{0.32}) and (\ref{0.33}) and letting $ j \rightarrow \infty$ gives
\[
\mu_M(x_1,x_2) \leq K_0 \; \mu_{V}  \big( \tilde{\psi}(x_1), \tilde{\psi}(x_2) \big).
\]
Since $ M $ is of Class II, we must have $ x_1 = x_2 $ whenever $ \tilde{\psi}(x_1) = \psi(x_2) $. It follows that $ \tilde{\psi} $ is injective on $ M $. The remainder of the argument proceeds exactly as in the proof of Theorem \ref{T1}, after making the relevant changes. Thus, in Case (b), $ M $ is $ K_0 $-quasiconformally equivalent to $ \Omega_{\infty} $. Since $ \Omega_{\infty} $ is conformally equivalent to the unit ball of $ \mathbb{R}^n $ via a M\"{o}bius transformation, the theorem follows.
\qed

\medskip

\noindent\textit{Proof of Theorem \ref{T2}(ii).}
We retain the notation from the proof of part (i).
As before, let $ \{U_n: n \ge 1\}$ be an exhaustion of $ M $ by open subsets such that $U_n \Subset U_{n+1} $ for all $ n \ge 1$. For each $ n $, there exists an $ N = N(n) \ge 1$ such that for every $j \geq N$, $U_n \subset M_j$ and hence $\psi_j$ is defined on $ U_n $. As $ \Omega $ is bounded, Theorem~\ref{normal2} guarantees the normality of $ \{\psi_j\vert_{U_n}\}_{j \geq N} $. It follows that the limit map $ \psi : M \rightarrow \overline{\Omega} $ is either constant or a $ K_0$-quasiconformal mapping of $ M $ into $ \Omega $. Recall that $ \Omega $, being a bounded domain in $ \mathbb{R}^n $, is of Class II, whereas $ M $ is of Class I. Therefore, by Theorem C (a) of \cite{F1}, there is no quasiconformal mapping from $ M $ into $ \Omega $. Consequently, the limit map must be constant, say $ \psi \equiv c \in \overline{\Omega} $. 

\medskip

We employ a similar construction as in the proof of Theorem 10.2 of \cite{F1}. Choose a sequence of neighbourhoods $ V_n = B(c, r_n)$ of $ c $ in $ \mathbb{R}^n $ with $ r_n \downarrow 0 $. The uniform convergence $ \psi_j \rightarrow c $ on the compact set $ \overline{U}_n $ implies that $ \psi_{j}( \overline{U}_n ) \subset V_n $ for all sufficiently large $j$. Since $ \overline{U}_n \subset M_{j} $ for all sufficiently large $j$, we may choose a strictly increasing sequence of integers $ \{ p_n \} $ such that $ \overline{U}_n \subset M_{p_n} $ and $ \psi_{p_n}( \overline{U}_n ) \subset V_n $ for every $ n$. Let $ h_n $ denote the affine map sending $ V_n $ onto the Euclidean unit ball in $ \mathbb{R}^n $. Composing $ h_n $ with the inverse stereographic projection, restricted to the unit ball, yields a $ 1$-quasiconformal mapping $ g_n $ from $ V_n $ into the $ n $-dimensional sphere $ S^n $. It follows that the maps $  g_n \circ \psi_{p_n}  $ are $ K_{p_n} $-quasiconformal on $ U_n $. Since $ \{U_n \} $ exhausts $ M $ and $ K_{p_n} \rightarrow K_0 $, the hypotheses of Theorem 9.3 of \cite{F1} are satisfied. Hence, $ M $ is $K_0$-quasiconformally equivalent to a domain of $ S^n $. Furthermore, since $ M $ is non-compact, this domain is necessarily a proper subdomain of $ S^n $. Via the stereographic projection, it is conformally equivalent to a domain $ G \subseteq \mathbb{R}^n $. As the property of being Class I is preserved under quasiconformal homeomorphisms, $ G $ is of Class I. 
\qed

\medskip

When $ K_0=1 $ and $ M $ is $C^{\infty}$-smooth, a stronger conclusion holds.

\begin{cor}
Let $ M $ be a $ C^{\infty}$-smooth $ n$-dimensional Riemannian manifold exhausted by an increasing sequence of open subsets $M_j \Subset M_{j+1}$ and let $ \Omega $ be a bounded domain in $ \mathbb{R}^n $. Assume that for each $ j $, there exists a $ K_j $-quasiconformal homeomorphism $\phi_j: \Omega \rightarrow M_j$ where $ K_j \rightarrow 1 $.

\begin{enumerate}

\item [(i)] If $ M $ is of Class II and $ \Omega $ has $C^2$-smooth boundary, then there exists a conformal $C^{\infty}$-diffeomorphism from $ M $ onto $ \Omega $ or the unit ball in $ \mathbb{R}^n $.    

\item [(ii)] If $ M $ is of Class I, then there exists a conformal $C^{\infty}$-diffeomorphism from $ M $ onto a Class I domain in $ \mathbb{R}^n $.
\end{enumerate}

\end{cor}

\begin{proof} In each of the two cases, it follows from Theorem \ref{T2} that the limit map is $ 1 $-quasiconformal. In other words, it is a conformal homeomorphism. By Theorem A of \cite{F4}, every $ 1$-quasiconformal mapping between $C^{\infty} $-smooth Riemannian manifolds is $ C^{\infty} $. Hence, the limit map is a conformal diffeomorphism of $ M $.     
\end{proof}

In Example 2 of \cite{G}, Gehring constructed a bounded domain in $ \mathbb{C}^n $, $ n \geq 3 $, which is not quasiconformally equivalent to the unit ball. This domain is $C^1$-smooth everywhere except at one point. Although this domain is quasiconformally homogeneous, it was observed in \cite{KV} that it does not admit a sequence of $K_j$-quasiconformal self maps, with $K_j$ uniformly bounded, that take an interior point to the isolated singular boundary point. Theorem \ref{T2} (i) illustrates another property of Gehring's domain, namely, that it cannot even be exhausted by quasiconformal images of the ball in $\mbb R^n$, with uniformly bounded dilatations.

\begin{thm}
Assume that in the Union Problem, $ M $ is a bounded domain in $ \mathbb{R}^n $ and $ \Omega = D \setminus A $, where $ D $  is a bounded domain in $ \mathbb{R}^n $ with $C^2$-smooth boundary and $ A \subset D $ is finite. Then $ M $ is $K$-quasiconformally equivalent to $ \Omega $ or to the unit ball in $ \mathbb{R}^n $.    
\end{thm}

\begin{proof} Let $ \phi_j : \Omega \rightarrow M_j $ be $ K$-quasiconformal homeomorphisms, and let $ \psi_j = \phi_j^{-1} $ be their inverses. Fix $ z_0 \in M $. As in the proof of Theorem~\ref{iso}, our goal is to show that the sequence $ \{ \psi_j(z_0) \} $ cannot have an accumulation point in $ A $.  

\medskip

Assume, to the contrary, that a subsequence of $ \{ \psi_j(z_0) \} $, which we continue to denote by the same symbol, converges to $ p_0 \in A $. Since $A $ is a finite set, it has  conformal capacity zero. Theorem 6.4.26 of \cite{GMP} now implies that $ \phi_j $ admits an extension 
\[
\Phi_j: D \rightarrow M_j \cup C_{\phi_j}(A),  
\]
where $ C_{\phi_j} (A) $ denotes the cluster set of $ \phi_j $ at $ A $. Moreover, $ \Phi_j $ is a $ K $-quasiconformal homeomorphism from $ D $ onto $ M_j \cup C_{\phi_j}(A) $. Furthermore, for every $ p \in A $, the cluster set $ C_{\phi_j}(p)$ is a singleton boundary point of $ M_j $ and $ \Phi_j(p) $ is precisely the unique point in $ C_{\phi_j}(p)$, and the cluster sets corresponding to distinct points of $ A $ are disjoint. In particular, $ \Phi_j(D) \subset \overline{M} $.

\medskip

Since $ M $ is bounded, each map $ \Phi_j $
omits two values of $\mathbb{R}^n \cup \{ \infty\} $
whose spherical distance is bounded below by a positive constant independent of $ j $. It follows from Theorems 19.2 and 20.5 of \cite{V} that the family $ \{ \Phi_j \} $ is normal. Passing to a subsequence if necessary, we may assume that $ \{ \Phi_j \} $ converges locally uniformly on $ D $ to a mapping $ \Phi : D \rightarrow \overline{M} $. Note that
\[
\Phi(p_0) = \lim_{j \rightarrow \infty} \Phi_j ( \psi_j(z_0) ) = \lim_{j \rightarrow \infty} \phi_j ( \psi_j(z_0) ) = z_0 \in M. 
\]
On the other hand, 
\[
\Phi(p_0) = \lim_{j\rightarrow \infty} \Phi_j(p_0) = \lim_{j\rightarrow \infty} q_j,
\]
where $ q_j = \Phi_j(p_0) \in \partial M_j $. 

\medskip

To conclude the proof, it suffices to show that the sequence $ \{q_j\} $ cannot be relatively compact in $ M $. Suppose, to the contrary, that $ \{q_j\} $ is relatively compact in $ M $. Since $ M = \cup M_j $, there exists $ k_0 \ge 1 $ such that $ q_j \in M_{k_0} $ for all $ j$ large. As the domains $ \{M_j \} $ are nested, we have $ M_{k_0} \subset M_j $ whenever $ j \geq k_0 $. Consequently, $ q_j \in M_j $ for each $ j \geq k_0 $, contradicting the fact that $ q_j \in \partial M_j $. This contradiction shows that $ \{ \psi_j(z_0) \} $ cannot have an accumulation point in $ A $, as claimed.

\medskip

Consequently, the sequence $ \{ \psi_j(z_0) \} $ can accumulate only on $ \partial D $ or at points of $ D \setminus A $. From this point onwards, the proof proceeds exactly as in Theorem \ref{T2}(i) yielding that $ M $ is $K$-quasiconformally equivalent to the unit ball in $ \mathbb{R}^n $ in the former case and to $ \Omega $ in the latter. 
\end{proof}

\end{document}